\documentclass{amsart}
\usepackage{amsmath}
\usepackage{amsthm}
\usepackage{amssymb}
\usepackage{xcolor}
\usepackage{enumerate}
\usepackage[colorlinks=true,linkcolor=blue,citecolor=blue,urlcolor=blue]{hyperref}
\usepackage{ulem}

\newtheorem{thm}{Theorem}
\newtheorem{lem}[thm]{Lemma}

\newtheorem{ass}[thm]{Assumption}
\newtheorem{defi}[thm]{Definition}
\newtheorem{prop}[thm]{Proposition}
\newtheorem{rk}[thm]{Remark}

\newcommand{\rr}{{\mathbb{R}}}
\newcommand{\WW}{{\mathbb{W}}}
\newcommand{\cF}{{\mathcal{F}}}
\newcommand{\cW}{{\mathcal{W}}}
\newcommand{\cP}{{\mathcal{P}}}
\newcommand{\cA}{{\mathcal{A}}}

\newcommand{\cC}{{\mathcal{C}}}
\newcommand{\cX}{{\mathcal{X}}}
\newcommand{\cM}{{\mathcal{M}}}
\newcommand{\cE}{{\mathcal{E}}}
\newcommand{\e}{\varepsilon}
\newcommand{\vip}{\vskip.12cm}
\newcommand{\indiq}{{\bf 1}}
\newcommand{\E}{\mathbb{E}}
\newcommand{\PP}{\mathbb{P}}
\newcommand{\bX}{\bar X}
\newcommand{\bU}{\bar U}
\newcommand{\bP}{\bar \PP}
\newcommand{\bE}{\bar \E}
\newcommand{\bOmega}{\bar \Omega}

\newcommand{\bcF}{\bar \cF}
\newcommand{\tX}{\tilde X}
\newcommand{\tY}{\tilde Y}
\newcommand{\tU}{\tilde U}
\newcommand{\vmark}{\xi}
\newcommand{\nmark}{\vmark'}
\newcommand{\emark}{\vmark^\dagger}
\newcommand{\pmark}{a}
\newcommand{\qmark}{b}
\newcommand{\rmark}{c}
\newcommand{\vmarkspace}{E}
\newcommand{\vmarkfield}{\cE}
\newcommand{\vmarklaw}{\pi}
\newcommand{\emarkspace}{\vmarkspace^\dagger}
\newcommand{\emarkfield}{\vmarkfield^\dagger}
\newcommand{\emarklaw}{\vmarklaw^\dagger}
\newcommand{\dd}{{\mathrm{d}}}

\begin{document}

\title[Mean-field theory via dissociated arrays]
{Mean-field theory via dissociated arrays for particle systems interacting through noisy weights}
\author{Nicolas Fournier}
\author{Datong Zhou}
\address{Nicolas Fournier, Sorbonne Universit\'e - LPSM (UMR 8001),
Campus Pierre et Marie Curie, Case courrier 158, 4 place Jussieu,
75252 Paris Cedex 05, France. \tt{nicolas.fournier@sorbonne-universite.fr}}
\address{Datong Zhou, Sorbonne Universit\'e - LPSM (UMR 8001) and LJLL (UMR 7598),
Campus Pierre et Marie Curie, Case courrier 158, 4 place Jussieu,
75252 Paris Cedex 05, France. \tt{datong.zhou@sorbonne-universite.fr}}

\begin{abstract}
We study a mean-field limit for a $N$-particle system in which each particle
follows a diffusion and interacts with other particles through a weight on each directed edge. Each weight evolves according to
its own nonlinear SDE driven by a Brownian motion, with coefficients involving the states of the two endpoint particles of the edge. The initial vertex and edge variables are assumed to have a dissociated Aldous--Hoover form. We construct the limiting nonlinear SDE by averaging the interaction over an independent neighbor and an edge input, prove its well-posedness, and show that the
dissociated vertex-edge structure is propagated by the dynamics. This
propagation property is an analogue of propagation of chaos in the case where the weight of each edge may remain correlated with the states of the two endpoint particles. Under either a bounded-observable assumption or a sub-Gaussian edge-input condition, the finite system converges to this limit through quantitative coupling estimates for a
typical particle and a typical edge. We also prove the convergence of the empirical
measure of particle's state pairs and their interaction weights.
\end{abstract}



\maketitle

\section{Description of the model and main result}

\subsection{Motivation}

In its classical exchangeable form, mean-field theory describes the
large-population limit of an interacting particle system through one typical particle, of which the state solves a nonlinear SDE, see the seminal works of Kac~\cite{kac1956foundations}, McKean~\cite{mckean1969propagation}, Sznitman~\cite{sznitman1991topics}, Méléard~\cite{mele}.
\vip
A first way to incorporate heterogeneity is to keep the interaction structure
fixed: the interaction kernel is prescribed in advance.
Graphon and graph-kernel mean-field limits describe such system, see for instance
Chiba and Medvedev~\cite{chiba2018mean}, Kaliuzhnyi-Verbovetskyi and Medvedev~\cite{kaliuzhnyi2018mean}, Bayraktar, Chakraborty and Wu~\cite{bayraktar2023graphon} and Jabin and co-authors~\cite{jabin2025mean,jabin2023mean,jabin2024non}.
In all these works, the interaction weights remain fixed along the dynamics.
\vip
However, in many adaptive network models, the interaction strength between
two particles is itself a dynamical variable, depending on the states of the two
particles and possibly driven by its own noise. Such feedback is natural in
co-evolutionary network models and in synaptic plasticity, see Gross and Blasius~\cite{gross2008adaptive} and Bi and Poo~\cite{bi2001synaptic}.
\vip
Mathematical works on genuine co-evolutionary dynamics include deterministic
continuum limits for pairwise adaptive networks, see Gkogkas, Kuehn and Xu~\cite{gkogkas2023continuum,gkogkas2025mean}, as well as probabilistic models
with intrinsic stochasticity and dynamically evolving edge variables, see Bayraktar and Wu~\cite{bayraktar2021mean} and Ganguly, Spiliopoulos and Sussman~\cite{ganguly2025mean}. The companion work
\cite{zhou2025non} develops a structural metric and sampling viewpoint for
linear weight adaptation.
\vip
The goal of the present paper is to treat a class of particle systems with
fully noisy and nonlinear real-valued edge weights. We keep track not only of a one-particle law, but of a vertex-edge array whose dissociated structure has to be propagated by the
dynamics.

\subsection{The interacting particle system}
We fix $d\geq 1$ and $m\geq 1$. For each $N\geq 2$, we set $I_N=\{1,\dots,N\}$ and consider the system
\begin{align}
X^{i,N}_t=&X^i_0+\int_0^t b(X^{i,N}_s)\dd s + \int_0^t \sigma(X^{i,N}_s)\dd B^i_s
+ \frac1{N-1} \sum_{j\in I_N,j\neq i} \int_0^t\phi(U^{ij,N}_s)\gamma(X^{i,N}_s,X^{j,N}_s) \dd s, \label{s1}\\
U^{ij,N}_t=&U^{ij}_0 +\int_0^t \alpha(X^{i,N}_s,X^{j,N}_s,U^{ij,N}_s)\dd s + \int_0^t
\beta(X^{i,N}_s,X^{j,N}_s,U^{ij,N}_s)\dd W^{ij}_s. \label{s2}
\end{align}
The first equation~\eqref{s1} has to hold for all $i \in I_N$, the second one~\eqref{s2} for all
$i,j\in I_N$ with $i\neq j$.

\subsection{Hypotheses}
Concerning randomness, we suppose the following.

\begin{ass}\label{asr}
The following random objects are independent:
\vip
\noindent $\bullet$ an i.i.d. family $(\vmark_i)_{i\geq 1}$ with common law $\vmarklaw$
valued in some measurable space $(\vmarkspace,\vmarkfield)$,
\vip
\noindent $\bullet$ an i.i.d. family $(\emark_{ij})_{i,j\geq 1, i\neq j}$ with common law $\emarklaw$
valued in some measurable space $(\emarkspace,\emarkfield)$,
\vip
\noindent $\bullet$ some i.i.d. $m$-dimensional Brownian motions $(B^i_t)_{t\geq 0, i\geq 1}$,
\vip
\noindent $\bullet$ some i.i.d. $1$-dimensional Brownian motions $(W^{ij}_t)_{t\geq 0, i,j\geq 1, i\neq j}$.
\vip
There are some measurable functions $F:\vmarkspace\to \rr^d$ and
$G:\vmarkspace^2\times{\emarkspace}\to \rr$ such that, for all $i\geq 1$,
$X_0^i=F(\vmark_i)$ and, for all $i,j\geq 1$ with $i\neq j$,
$U_0^{ij}=G(\vmark_i,\vmark_j,\emark_{ij})$.
\end{ass}

This is the standard dissociated Aldous--Hoover form for the
initial edge weights: each edge $ij$ has its own mark $\emark_{ij}$, independent of the vertex variables $\xi_i$ and $\xi_j$ and of the other edge marks. For 
the representation-theoretic background, see Aldous~\cite{aldous1981representations}, Hoover~\cite{hoover1979relations} and Kallenberg\cite{kallenberg1989representation,kallenberg2005probabilistic};
for the connection with graph limits, see Janson and Diaconis~\cite{janson2008graph}. In the
present nonlinear setting, the edge mark is not merely a technical decoration:
it records edge-level randomness which can be seen by nonlinear functions of
the weight and which later evolves together with the edge Brownian motion.

\vip
We will assume at least the following moment condition.
\begin{ass}\label{asm}
For $\vmark_1,\vmark_2,\emark_{12}$ independent with $\vmark_1,\vmark_2\sim \vmarklaw$ and $\emark_{12}\sim\emarklaw$, it holds that 
$$
\E[|F(\vmark_1)|^2+(G(\vmark_1,\vmark_2,\emark_{12}))^2]<\infty.
$$
\end{ass}

Concerning the coefficients, we assume the following.

\begin{ass}\label{asc}
The functions $b:\rr^d\to \rr^d$, $\sigma:\rr^d \to \cM_{d\times m}(\rr)$,
$\gamma:\rr^d\times\rr^d\to \rr^d$,
$\alpha: \rr^d\times\rr^d\times \rr \to \rr$ and $\beta: \rr^d\times\rr^d\times \rr \to \rr$
are bounded and globally Lipschitz continuous, while $\phi:\rr\to\rr$ is globally Lipschitz continuous.
\end{ass}

We keep $\phi$ merely Lipschitz at this stage, in order to include the natural
choice $\phi(u)=u$. Additional boundedness or moment assumptions are stated
explicitly in the results below.

\begin{rk}
Under Assumptions~\ref{asr}, \ref{asm} and~\ref{asc},
the system \eqref{s1}-\eqref{s2} has a pathwise unique strong solution,
because it can be written as an SDE in $\rr^{Nd+N(N-1)}$ with locally Lipschitz continuous coefficients
with at most linear growth.
\end{rk}

\subsection{The mean-field limit}

We aim to study what happens as $N$ tends to infinity.
The mean-field limit is described by a nonlinear SDE coupled with an averaged
edge interaction.
To rigorously define a limit process, we need the following result, that will be proved in Section~\ref{smes}
using some arguments found in Karandikar~\cite{k}. It provides a measurable
solution map for the edge equation.

\begin{lem}\label{mes}
Grant Assumptions~\ref{asr}, \ref{asm} and~\ref{asc} and set $\cC_\ell=C(\rr_+,\rr^\ell)$. There is a measurable map
$$
\Gamma=(\Gamma_t)_{t\geq 0} : \vmarkspace \times \cC_d\times \vmarkspace \times \cC_d\times\emarkspace\times\cC_1 \to \cC_1
$$
such that for any space $(\Omega,\cF,(\cF_t)_{t\ge 0},\PP)$, any triple $(\vmark,\nmark,\emark)$ of independent $\cF_0$-measurable variables with $\vmark,\nmark\sim\vmarklaw$ and $\emark\sim\emarklaw$,
any pair of continuous $(\cF_t)_{t\geq 0}$-adapted
$\rr^d$-valued processes $X=(X_t)_{t\geq 0}$, $Y=(Y_t)_{t\geq 0}$ such that
\begin{equation}\label{condA}
\text{$\forall\; T>0, \; \exists \; A_T>0$ such that $\forall \;0\leq s<t<s+1\leq T+1$,}\quad \E[|X_t-X_s|^2]\leq
A_T (t-s),
\end{equation}
and such that $Y$ satisfies the same condition,
any $1$-dimensional $(\cF_t)_{t\geq 0}$-Brownian motion $W=(W_t)_{t\geq 0}$,
the pathwise unique solution $U=(U_t)_{t\geq 0}$ to
\begin{equation}\label{dfg}
U_t=G(\vmark,\nmark,\emark)+ \int_0^t \alpha(X_s,Y_s,U_s)\dd s + \int_0^t \beta(X_s,Y_s,U_s)\dd W_s
\end{equation}
satisfies $U=\Gamma(\vmark,X,\nmark,Y,\emark,W)$ in the sense that a.s., for all $t\geq 0$,
$U_t=\Gamma_t(\vmark,X,\nmark,Y,\emark,W)$.
One may moreover build $\Gamma$ in such a way that
for all random sextuplet $(\vmark,X,\nmark,Y,\emark,W)$ as above, a.s.,
\begin{equation}\label{prog}
\text{for all $t\geq 0$,} \quad
\Gamma_t(\vmark,X,\nmark,Y,\emark,W)=\Gamma_t(\vmark,(X_{s\land t})_{s\geq 0},\nmark,
Y,\emark,W).
\end{equation}
\end{lem}

We call a pair $(\vmark,X)$ admissible if it is defined on some
filtered probability space, if $\vmark\sim\vmarklaw$ is $\cF_0$-measurable, and if $X=(X_t)_{t\geq0}$ is continuous, adapted and satisfies~\eqref{condA}. A
probability measure $\mu$ on $\vmarkspace\times\cC_d$ is called admissible if it is
the law of an admissible pair.
\vip

The next lemma, to be proved in Section~\ref{smes}, introduces the averaged edge-interaction operator. Informally, this operator takes the edge contribution seen from an
admissible input $(\vmark,X)$, computed through the map $\Gamma$ from
Lemma~\ref{mes} and the nonlinearity $\phi$, and averages it over an independent
neighbor with law $\mu$ and over the fixed edge input law $\emarklaw\otimes\WW$,
where $\WW$ is the Wiener measure on $\cC_1$.

\begin{lem}\label{Lambda}
Grant Assumptions~\ref{asr}, \ref{asm} and~\ref{asc} and suppose either that $\phi$ is bounded or that
\begin{equation}\label{m2c}
\sup_{\pmark,\qmark\in \vmarkspace}\int_{\emarkspace}|G(\pmark,\qmark,\rmark)|^2\emarklaw(\dd \rmark)<\infty.
\end{equation}
For any admissible probability measure $\mu$ on $\vmarkspace\times \cC_d$, we use the
pointwise notation, for $t\geq 0$ and $(\pmark,x)\in \vmarkspace\times \cC_d$,
$$
\Lambda_\mu(t,\pmark,x)=\int_{\vmarkspace\times \cC_d} \int_{\emarkspace\times \cC_1}
\phi(\Gamma_t(\pmark,x,\qmark,y,\rmark,w))\gamma(x_t,y_t) \emarklaw(\dd \rmark)\WW(\dd w)
\mu(\dd \qmark,\dd y).
$$

(i) For each admissible pair $(\vmark,X)$, the map
$t\mapsto \Lambda_\mu(t,\vmark,X)$ is a.s. continuous.
\vip
(ii) If $(\vmark,X)$ is admissible with respect to $(\cF_t)_{t\ge 0}$, then
$(\Lambda_\mu(t,\vmark,X))_{t\geq 0}$ is $(\cF_t)_{t\geq 0}$-adapted.
\end{lem}

For $(\xi,X)$ an admissible pair, we use the 
shorthand
$$
\Lambda_{\vmark,X}(t,a,x):=\Lambda_{\mathrm{Law}(\vmark,X)}(t,a,x).
$$
In particular, $\Lambda_{\vmark,X}(t,\vmark,X)$ means that the law in the operator is
$\mathrm{Law}(\vmark,X)$ and that the operator is evaluated at
$(\vmark,X)$.
We now introduce the limiting nonlinear SDE corresponding to~\eqref{s1}-\eqref{s2}.

\begin{defi}
Grant Assumptions~\ref{asr}, \ref{asm} and~\ref{asc} and suppose either that $\phi$ is bounded or that~\eqref{m2c} holds.
We say that $(\vmark,X)$ defined on some $(\Omega,\cF,(\cF_t)_{t\ge 0},\PP)$
solves the nonlinear SDE if  $\vmark\sim \vmarklaw$ is $\cF_0$-measurable, if $X=(X_t)_{t\geq 0}$
is continuous, $\rr^d$-valued, $(\cF_t)_{t\geq 0}$-adapted and satisfies~\eqref{condA}, and
if there is a $m$-dimensional $(\cF_t)_{t\geq 0}$-Brownian motion
$B=(B_t)_{t\geq 0}$ such that a.s., for all $t\geq 0$,
\begin{equation}
X_t=F(\vmark)+ \int_0^t b(X_s)\dd s + \int_0^t \sigma(X_s)\dd B_s + \int_0^t \Lambda_{\vmark,X}(s,\vmark,X)\dd s.
\end{equation}
\end{defi}

In words, $X$ is the limiting trajectory of a typical particle. Its dynamics
contains the local drift $b$, the Brownian noise $\sigma dB$, and the averaged
interaction field $\Lambda_{\vmark,X}$, which is obtained by averaging over an
independent typical neighbor, together with an additional independent edge
variable and an independent edge Brownian motion.

\subsection{Results}

Our first main result establishes well-posedness for the limiting equation before
any finite-$N$ comparison is made. The averaged operator $\Lambda_{\vmark,X}$
depends on the law of the candidate process itself, so the statement is a McKean--Vlasov well-posedness result for the vertex component.

\begin{thm}\label{mr} Grant Assumptions~\ref{asr}, \ref{asm} and~\ref{asc}
and suppose that either $\phi$ is bounded or~\eqref{m2c}.
\vip
(i) For any stochastic basis $(\Omega,\cF,(\cF_t)_{t\geq 0},\PP)$, any
$\cF_0$-measurable random variable $\vmark$ with law $\vmarklaw$, and any
$m$-dimensional $(\cF_t)_{t\geq 0}$-Brownian motion $B$, there exists a
pathwise unique continuous adapted process $X$ such that $(\vmark,X)$ solves the
nonlinear SDE driven by $B$.
\vip
(ii) Uniqueness in law holds: if $(\vmark,X)$ and $(\tilde \vmark,\tilde X)$ are two
solutions, possibly defined on different stochastic bases, then
$\mathrm{Law}(\vmark,X)=\mathrm{Law}(\tilde \vmark,\tilde X)$.
\end{thm}

Our second result concerns the limit $N\to \infty$.
The comparison process is obtained by solving the nonlinear SDE for each
vertex input $(\vmark_i,B^i)$ and then, for each directed edge, solving the edge
equation driven by the corresponding edge input $(\emark_{ij},W^{ij})$. Thus
the limiting object is an explicit dissociated vertex-edge array, built on the
same sources of randomness as the particle system. The theorem first gives
pathwise coupling estimates for a typical particle and a typical edge, and then
projects this coupling to the empirical law of triples
$(X^{i,N}_t,X^{j,N}_t,U^{ij,N}_t)$.

\begin{thm}\label{mr2}
Grant Assumptions~\ref{asr}, \ref{asm} and~\ref{asc} and suppose that either
$\phi$ is bounded or~\eqref{m2c} and
\begin{equation}\label{meu}
\text{there is $\theta_0>0$ such that}\quad
\int_{\vmarkspace^2\times \emarkspace}
\exp\big(\theta_0 |G(\pmark,\qmark,\rmark)|^2\big)
\vmarklaw(\dd \pmark) \vmarklaw(\dd \qmark)
\emarklaw(\dd \rmark)<\infty.
\end{equation}
For each $i\geq 1$, consider the solution $\bX^i=(\bX^i_t)_{t\geq 0}$ to the nonlinear SDE corresponding to
$\vmark_i$ and $B^i$. For $i,j\geq 1$ with $i\neq j$, consider the pathwise unique solution $\bU^{ij}=(\bU^{ij}_t)_{t\geq 0}$ to
\begin{equation}\label{dfg2}
\bU^{ij}_t=G(\vmark_i,\vmark_j ,\emark_{ij})
+ \int_0^t \alpha(\bX^i_s,\bX^j_s,\bU^{ij}_s)\dd s + \int_0^t \beta(\bX^i_s,\bX^j_s,\bU^{ij}_s)\dd W_s^{ij}.
\end{equation}

(i) If $\phi$ is bounded, then for all $T>0$, there is a constant $C_T>0$ such that
for all $N\geq 2$,
$$
\E\Big[\sup_{t\in[0,T]}(|X^{1,N}_t-\bX^1_t|^2+|U^{12,N}_t-\bU^{12}_t|^2)\Big] \leq \frac{C_T}N.
$$

(ii) If~\eqref{m2c} and~\eqref{meu} hold, then
for all $T>0$, there is  $C_{T}>0$ such that
for all $N\geq 2$,
$$
\E\Big[\sup_{t\in[0,T]}(|X^{1,N}_t-\bX^1_t|+|U^{12,N}_t-\bU^{12}_t|)\Big] \leq \frac{C_T\exp(C_T\sqrt{\log N})}{N^{\frac 12}}.
$$

(iii) In either case above, introduce
$f_t=\mathrm{Law}(\bX^1_t,\bX^2_t,\bU^{12}_t)$.
For all $T>0$, there is  $C_T>0$ such that for all $N\geq 3$,
$$
\sup_{t\in [0,T]}\E\Big[\cW_1\Big(\frac1{N(N-1)}\sum_{i,j \in I_N,i\neq j}\delta_{(X^{i,N}_t,X^{j,N}_t,U^{ij,N}_t)},f_t\Big)\Big]
\leq \frac{C_T}{N^{\frac 1{2d+1}}}.
$$
\end{thm}

The Wasserstein distance $\cW_1$ between two probability measures $f,g$ on $\rr^d\times\rr^d\times \rr$ is defined
by
$$
\cW_1(f,g)=\inf\Big\{\E[|X-\tX|+|Y-\tY|+|U-\tU|] : (X,Y,U)\sim f, (\tX,\tY,\tU)\sim g\Big\}.
$$

The first propagation estimate is in $L^2$ and uses the boundedness of $\phi$.
When $\phi$ is unbounded but Lipschitz, the proof instead uses a
sub-Gaussian control of the edge variable and yields the weaker $L^1$ rate in
(ii). The same truncation strategy could be adapted
to an $L^2$ estimate, but this would produce a much
worse rate than in the $L^1$ estimate above. We therefore state
only the $L^1$ bound.
Point (iii) then translates these couplings into the
edge time-marginals analogue of classical propagation of chaos.

\subsection{Propagation of dissociatedness}
In the classical setting, propagation of chaos is
represented by independent copies of a single nonlinear process. Here the
limiting object is an array: subarrays supported on disjoint vertex sets are
independent, but an edge remains coupled with its two endpoints. The next remark
makes explicit how this Aldous--Hoover structure is propagated by the limiting
dynamics.

\begin{rk}\label{rk:mr2-coupling}
After fixing measurable versions of the strong solution maps, we
may write
$$
\bX^i=\mathsf{F}(\vmark_i,B^i),\qquad
\bU^{ij}=\mathsf{G}(\vmark_i,B^i,\vmark_j,B^j,\emark_{ij},W^{ij}),
$$
where $\mathsf{F}:\vmarkspace\times\cC_m\to\cC_d$ is the solution map for the
nonlinear SDE and
$$
\mathsf{G}(\pmark,v,\qmark,\tilde v,\rmark,w)
=\Gamma(\pmark,\mathsf{F}(\pmark,v),\qmark,\mathsf{F}(\qmark,\tilde v),\rmark,w).
$$
The dissociatedness (Aldous--Hoover representation) of Assumption~\ref{asr} is therefore propagated by
the limit dynamics: vertices carry $(\vmark_i,B^i)$ and directed edges carry
$(\emark_{ij},W^{ij})$. Hence, for any $t\geq 0$, for any collection 
$C\subset \{(i,j) : i\geq 1, j\geq 1, i\neq j\}$ such that for all $(i,j),(i',j') \in C$ distinct, the four numbers $i,j,i',j'$ are distinct, the variables $(\bX^i_t,\bX^j_t,\bU^{ij}_t)_{(i,j)\in C}$ are
i.i.d. This propagation of dissociatedness is the analogue of the classical propagation of chaos.
\end{rk}

\subsection{Comments}\label{dis}

One line of non-exchangeable
mean-field theory starts from a fixed graph kernel or graphon and derives an
extended Vlasov or McKean--Vlasov equation, see for
instance the works cited above~\cite{chiba2018mean,kaliuzhnyi2018mean,bayraktar2023graphon,jabin2025mean,jabin2023mean}.
This is particularly well-suited to non-adaptive interaction structures,
with fixed interaction weights. 
Here the fixed-label picture means the following type of representation: one
embeds the vertices into a predetermined latent space $D$ and represents the
edge structure at time $t$ by a deterministic kernel
$w_t:D\times D\to \rr$, evaluated at the two labels. The label of a vertex is
chosen once and for all, and the random evolution is not allowed to enlarge the
domain by adding new vertex or edge randomness.
\vip

A second line studies dynamic random networks using exchangeable arrays and
graph limits, closely connected with the Aldous--Hoover viewpoint; see
Crane~\cite{crane2016dynamic,crane2017exchangeable}, Černỳ and Klimovsky~\cite{cerny2020markovian}, Athreya, Den Hollander and Röllin~\cite{athreya2021graphon}
and Braunsteins, den Hollander and Mandjes~\cite{braunsteins2022graphon,braunsteins2023sample}.
These papers actually study some graph-valued Markov processes. There are no underlying  particles.

\vip
Closer to our setting are works on genuine co-evolutionary dynamics, where the
edge variables and the vertex states influence each other. In the deterministic
case, Gkogkas, Kuehn and Xu~\cite{gkogkas2023continuum,gkogkas2025mean}
establish continuum and mean-field limits for adaptive Kuramoto-type networks
through evolving kernels, graph measures, or signed digraph measures. These
models share the pairwise feedback mechanism with ours, but they are fully
deterministic systems.

\vip
Probabilistic models with intrinsic stochasticity have also been considered.
Bayraktar and Wu~\cite{bayraktar2021mean} study a continuous-time model in
which both the vertex and edge states take values in some countable spaces, while
Ganguly, Spiliopoulos and Sussman~\cite{ganguly2025mean} analyze a discrete-time
latent-variable model with binary dynamic edges, feedback effects, and
graphon/multiplexon limits.

\vip
The companion work \cite{zhou2025non} develops a structural metric and sampling
viewpoint for the same system~\eqref{s1}-\eqref{s2} when the weight adaptation is
linear and non noisy (that is when  $\phi(u)=u$, $\alpha(x,y,u)=a(x,y)+b(x,y)u$ and $\beta(x,y,u)=0$). In that setting, weight fluctuations can be absorbed at the level of the empirical
state-weight structure. The main stability estimate in~\cite{zhou2025non} controls the
evolution in a hybrid Wasserstein and cut distance by the initial structural
error and the graphon sampling error.

\vip
Here we allow a more general edge dynamics,
where the averaging reduction is no longer available. The edge variables solve their own noisy dynamics and enter
the particle equation nonlinearly, so the full edge law and its correlation with
the endpoint trajectories have to be retained in the mean-field description. To
our knowledge, this is the first mean-field result for fully
nonlinear and noisy adaptive weights.

\subsection{Technical comments}
We finish with a few comments on possible extensions and on the rates obtained
above. First, it would be natural to relax the boundedness assumptions on the
coefficients, while keeping global Lipschitz continuity. This should be possible
under suitable additional moment assumptions, but would require more involved
estimates. 
\vip

The convergence rate in Theorem~\ref{mr2}-(iii) is not intended to be optimal.
Our proof partitions the directed edges, so that each independent
subsample has only order $N$ (independent) elements. Applying the standard empirical
$\cW_1$ estimate in dimension $2d+1$, see~\cite[Theorem~1]{fg}, gives the rate
$N^{-1/(2d+1)}$. In the idealized situation where the $N(N-1)$ edge
triples were independent, the same general empirical $\cW_1$ estimate
would give a rate of order $N^{-2/(2d+1)}$. We do not know what the optimal rate should be.

\vip

A related extension would be to replace the diffusive particle dynamics by
jump-type dynamics, as in integrate-and-fire models from mathematical
neuroscience. Such models are often driven by Poisson events rather than
Brownian noises, see for instance \cite{jabin2024non,jabin2023mean}. We expect
that the same dissociated vertex-edge viewpoint should remain useful.

\subsection{Plan of the paper}

Section~2 collects the estimates needed for the nonlinear operator $\Lambda$ and then proves the well-posedness and propagation results. Section~3 proves the measurability and non-anticipativity statements for the edge solution map $\Gamma$ and the averaged operator $\Lambda$.

\section{Main proofs}

\subsection{Estimates}

The purpose of this
subsection is to keep the inequality estimates separate from the fixed-point and
propagation arguments. We use the map
$\Gamma$ from Lemma~\ref{mes} and the definition of
$\Lambda_\mu$; the estimates collected here will then be used in the proof of
Lemma~\ref{Lambda}. We first record estimates for the edge map $\Gamma$, with
one auxiliary space for the edge mark $\emark$ and another one for
the Brownian input $W$. We then derive the corresponding estimates for the
averaged operator $\Lambda_\mu$.

\begin{prop}\label{estgamma}
Grant Assumptions~\ref{asr}, \ref{asm} and~\ref{asc}. For all $T>0$, there is
$C_T>0$, depending only on $T,\alpha,\beta$, such that the
following estimates hold.
\vip
(i) For every pair of independent admissible inputs $(\vmark,X)$ and $(\nmark,Y)$ defined on a probability space $(\Omega,\cF,\PP)$,
every auxiliary edge-mark probability space
$(\Omega^\dagger,\cF^\dagger,\PP^\dagger)$ carrying an
$\emarkspace$-valued random variable $\emark$ with law
$\emarklaw$, and every auxiliary Brownian space
$(\Omega^\#,\cF^\#,(\cF_t^\#)_{t\geq0},\PP^\#)$ carrying a one-dimensional
Brownian motion $W$, all independent of each other,
\begin{align*}
\E^\dagger\Big[\E^\#
\Big[\sup_{t\in[0,T]}|\Gamma_t(\vmark,X,\nmark,Y,\emark,W)|^2\Big]
\Big]\leq 2\E^\dagger\big[|G(\vmark,\nmark,\emark)|^2\big]+C_T.
\end{align*}
\vip
(ii) For every pair of vertex-side triples $(\vmark,X,\tilde X)$ and
$(\nmark,Y,\tilde Y)$ such that $(\vmark,X)$, $(\vmark,\tilde X)$, $(\nmark,Y)$ and
$(\nmark,\tilde Y)$ are admissible and such that $(\vmark,X,\tilde X)$ is independent of $(\nmark,Y,\tilde Y)$, and every pair of auxiliary spaces as in (i),
independent of these triples, for all $t\in[0,T]$,
\begin{align*}
&\E^\dagger\Big[\E^\#
\Big[\sup_{u\in[0,t]}|\Gamma_u(\vmark,X,\nmark,Y,\emark,W)
-\Gamma_u(\vmark,\tilde X,\nmark,\tilde Y,\emark,W)|^2\Big]\Big]\\
\leq& C_T\Big(\sup_{s\in[0,t]}|X_s-\tilde X_s|^2
+\sup_{s\in[0,t]}|Y_s-\tilde Y_s|^2\Big).
\end{align*}
\end{prop}

\begin{proof}
For (i), set
$U_t=\Gamma_t(\vmark,X,\nmark,Y,\emark,W)$, $t\geq0$. Lemma~\ref{mes} gives
$$
U_t=G(\vmark,\nmark,\emark)+\int_0^t \alpha(X_s,Y_s,U_s)\dd s
+\int_0^t \beta(X_s,Y_s,U_s)\dd W_s.
$$
Since $\alpha$ and $\beta$ are bounded, the Cauchy--Schwarz inequality for the
drift term, and the Burkholder-Davies-Gundy inequality for the
martingale term, give
$$
\E^\#\Big[\sup_{t\in[0,T]}|U_t-G(\vmark,\nmark,\emark)|^2\Big]\leq C_T.
$$
Hence
$$
\E^\#\Big[\sup_{t\in[0,T]}|U_t|^2\Big]\leq
2|G(\vmark,\nmark,\emark)|^2+C_T,
$$
and applying $\E^\dagger$ allows us to conclude.
\vip
For (ii), set
$$
U_r=\Gamma_r(\vmark,X,\nmark,Y,\emark,W),\qquad
\tU_r=\Gamma_r(\vmark,\tilde X,\nmark,\tilde Y,\emark,W),\qquad r\geq0.
$$
Recalling Lemma~\ref{mes}, we see that
\begin{align*}
U_r-\tU_r=&\int_0^r[\alpha(X_s,Y_s,U_s)-\alpha(\tilde X_s,\tilde Y_s,\tU_s)]\dd s
+\int_0^r[\beta(X_s,Y_s,U_s)-\beta(\tilde X_s,\tilde Y_s,\tU_s)]\dd W_s.
\end{align*}
Using the Lipschitz continuity of $\alpha,\beta$, the Cauchy--Schwarz inequality
for the drift term, and the Burkholder-Davies-Gundy inequality for the martingale term, we get, for $r\leq T$,
\begin{align*}
\E^\#\Big[\sup_{u\in[0,r]}|U_u-\tU_u|^2\Big]
&\leq C_T\int_0^r\Big(|X_s-\tilde X_s|^2+|Y_s-\tilde Y_s|^2 +\E^\#\Big[\sup_{u\in[0,s]}|U_u-\tU_u|^2\Big]\Big)\dd s.
\end{align*}
Thanks to Gronwall's lemma,
\begin{align*}
\E^\#\Big[\sup_{u\in[0,r]}|U_u-\tU_u|^2\Big]
&\leq C_T\Big(\sup_{s\in[0,t]}|X_s-\tilde X_s|^2
+\sup_{s\in[0,t]}|Y_s-\tilde Y_s|^2\Big).
\end{align*}
Applying $\E^\dagger$ allows us to conclude.
\end{proof}

\begin{prop}\label{estlambda}
Grant Assumptions~\ref{asr}, \ref{asm} and~\ref{asc} and suppose either that
$\phi$ is bounded or that~\eqref{m2c} holds.
For all $T>0$, there is $C_T>0$ such that the following estimates hold.
\vip
(i) For any admissible probability measure $\mu$ on
$\vmarkspace\times\cC_d$, any admissible pair $(\vmark,X)$, any
$t\in[0,T]$,
\begin{equation*}
|\Lambda_\mu(t,\vmark,X)|\leq C_T\quad \hbox{a.s.}
\end{equation*}
\vip
(ii) Let $(\Omega,\cF,(\cF_t)_{t\geq0},\PP)$ support an $\cF_0$-measurable random
variable $\vmark$ and continuous adapted processes $X=(X_t)_{t\geq0}$ and
$\tX=(\tX_t)_{t\geq0}$ such that $(\vmark,X)$ and $(\vmark,\tX)$ are admissible.
Write $\Lambda_{\vmark,X}:=\Lambda_{\mathrm{Law}(\vmark,X)}$ and
$\Lambda_{\vmark,\tX}:=\Lambda_{\mathrm{Law}(\vmark,\tX)}$. Then, for all
$t\in[0,T]$,
\begin{equation*}
\E\Big[|\Lambda_{\vmark,X}(t,\vmark,X)-\Lambda_{\vmark,\tX}(t,\vmark,\tX)|^2\Big]
\leq C_T\E\Big[\sup_{s\in[0,t]}|X_s-\tX_s|^2\Big].
\end{equation*}
\end{prop}

\begin{proof}
We start with (i). Additionally to $(\Omega,\cF,(\cF_t)_{t\geq0},\PP)$ on which is defined $(\xi,X)$, we introduce a probability space
$(\Omega^*,\cF^*,(\cF_t^*)_{t\geq0},\PP^*)$ endowed with an admissible pair
$(\nmark,Y)$ with law $\mu$, a probability space
$(\Omega^\dagger,\cF^\dagger,\PP^\dagger)$ endowed with an
$\emarkspace$-valued random variable $\emark$ with
law $\emarklaw$, and a probability space
$(\Omega^\#,\cF^\#,(\cF_t^\#)_{t\geq0},\PP^\#)$ endowed with a one-dimensional
Brownian motion $W$. The three auxiliary spaces are taken independent of $(\Omega,\cF,(\cF_t)_{t\geq0},\PP)$ and of each other. Recalling Lemma~\ref{Lambda},
$$
\Lambda_{\mu}(t,\xi,X)=\E^*\Big[\E^\dagger\Big[\E^\#\Big[
\phi(\Gamma_t(\vmark,X,\nmark,Y,\emark,W)) \gamma(X_t,Y_t)
\Big]\Big]\Big].
$$
Define
\begin{equation*}
\Psi_{\mu,T}(\vmark,X):=
\E^*\Big[\E^\dagger\Big[\E^\#\Big[
\sup_{t\in [0,T]} |\phi(\Gamma_t(\vmark,X,\nmark,Y,\emark,W))|
\Big]\Big]\Big].
\end{equation*}
If $\phi$ is bounded, then $\Psi_{\mu,T}(\vmark,X)\leq ||\phi||_\infty$. If~\eqref{m2c} holds, then we use that $|\phi(u)|\leq C(1+|u|)$, Proposition~\ref{estgamma}-(i), together with the Cauchy--Schwarz inequality. This gives \begin{align*}
\Psi_{\mu,T}(\vmark,X)
&\leq C+C\E^*\Big[\E^\dagger\Big[\E^\#
\Big[\sup_{t\in[0,T]}|\Gamma_t(\vmark,X,\nmark,Y,\emark,W)|\Big]\Big]\Big]
\leq C+ C_T \E^*[\E^\dagger[|G(\xi,\xi',\xi^\dagger)|^2]]^{\frac12}.
\end{align*}
This last quantity is bounded by some constant $C_T$ thanks to~\eqref{m2c}.
\vip

Then~(i) follows from
$|\Lambda_\mu(t,\vmark,X)|\leq ||\gamma||_\infty\Psi_{\mu,T}(\vmark,X)$.
\vip
We now prove (ii). Additionally to $(\Omega,\cF,(\cF_t)_{t\geq0},\PP)$, we
introduce a second probability space
$(\Omega^*,\cF^*,(\cF_t^*)_{t\geq0},\PP^*)$ endowed with a copy
$(\vmark^*,X^*,\tX^*)$ of $(\vmark,X,\tX)$, as well as a third probability space
$(\Omega^\dagger,\cF^\dagger,\PP^\dagger)$ endowed with an
$\emarkspace$-valued random variable $\emark$ with
law $\emarklaw$ and a fourth probability space
$(\Omega^\#,\cF^\#,(\cF_t^\#)_{t\geq0},\PP^\#)$ endowed with a one-dimensional
Brownian motion $W$. On
$$
(\bOmega,\bcF,(\bcF_t)_{t\geq0},\bP)
=\big(\Omega\times\Omega^*\times\Omega^\dagger\times\Omega^\#,
\cF\otimes\cF^*\otimes\cF^\dagger\otimes\cF^\#,
(\cF_t\otimes\cF_t^*\otimes\cF^\dagger\otimes\cF_t^\#)_{t\geq0},
\PP\otimes\PP^*\otimes\PP^\dagger\otimes\PP^\#\big),
$$
the four families $(\vmark,X,\tX)$, $(\vmark^*,X^*,\tX^*)$,
$\emark$ and $W$ are independent. We consider the pathwise unique
solutions $U$ and $\tU$ to
\begin{align*}
U_t=&G(\vmark,\vmark^*,\emark)
+\int_0^t \alpha(X_s,X_s^*,U_s)\dd s
+\int_0^t \beta(X_s,X_s^*,U_s)\dd W_s,\\
\tU_t=&G(\vmark,\vmark^*,\emark)
+\int_0^t \alpha(\tX_s,\tX_s^*,\tU_s)\dd s
+\int_0^t \beta(\tX_s,\tX_s^*,\tU_s)\dd W_s.
\end{align*}
By Lemma~\ref{mes},
$U=\Gamma(\vmark,X,\vmark^*,X^*,\emark,W)$ and
$\tU=\Gamma(\vmark,\tX,\vmark^*,\tX^*,\emark,W)$. Thus, as in (i),
$$
\Lambda_{\vmark,X}(t,\vmark,X)
=\E^*\Big[\E^\dagger\Big[\E^\#\big[\phi(\Gamma_t(\vmark,X,\vmark^*,X^*,\emark,W))\gamma(X_t,X_t^*)\big]\Big]\Big]
=\E^*\Big[\E^\dagger\Big[\E^\#\big[\phi(U_t)\gamma(X_t,X_t^*)\big]\Big]\Big].
$$
Similarly,
$\Lambda_{\vmark,\tX}(t,\vmark,\tX)
=\E^*[\E^\dagger[\E^\#[\phi(\tU_t)\gamma(\tX_t,\tX_t^*)]]]$, so that
\begin{align}
\E\Big[|\Lambda_{\vmark,X}(t,\vmark,X)-\Lambda_{\vmark,\tX}(t,\vmark,\tX)|^2\Big]
=& \E\Big[\Big|\E^*\Big[\E^\dagger\Big[\E^\#\Big[\phi(U_t)\gamma(X_t,X_t^*)
-\phi(\tU_t)\gamma(\tX_t,\tX_t^*)\Big]\Big]\Big]\Big|^2\Big]\notag\\
\leq& \E\Big[\E^*\Big[\E^\dagger\Big[\E^\#\Big[
\Big|\phi(U_t)\gamma(X_t,X_t^*)
-\phi(\tU_t)\gamma(\tX_t,\tX_t^*)\Big|^2\Big]\Big]\Big]\Big]\notag\\
=& \bE\Big[\Big|\phi(U_t)\gamma(X_t,X_t^*)
-\phi(\tU_t)\gamma(\tX_t,\tX_t^*)\Big|^2\Big]. \label{lpd}
\end{align}
\vip

{\it Case 1: $\phi$ bounded.} Using that $\phi,\gamma$ are bounded and Lipschitz
continuous, \eqref{lpd} implies that
\begin{align}\label{1jab1}
\E\Big[|\Lambda_{\vmark,X}(t,\vmark,X)-\Lambda_{\vmark,\tX}(t,\vmark,\tX)|^2\Big]
\leq& C\bE[|U_t-\tU_t|^2]
+C\bE[|X_t-\tX_t|^2+|X_t^*-\tX_t^*|^2].
\end{align}
By Proposition~\ref{estgamma}-(ii),
\begin{equation}\label{1jab2}
\bE[|U_t-\tU_t|^2]\leq
C_T\bE\Big[\sup_{s\in[0,t]}\big(|X_s-\tX_s|^2+|X_s^*-\tX_s^*|^2\big)\Big].
\end{equation}
Inserting~\eqref{1jab2} into~\eqref{1jab1}, we find
\begin{align*}
\E\Big[|\Lambda_{\vmark,X}(t,\vmark,X)\!-\!\Lambda_{\vmark,\tX}(t,\vmark,\tX)|^2\Big]
\leq C_T\bE\Big[\sup_{s\in[0,t]}
\big(|X_s\!-\!\tX_s|^2\!+\!|X_s^*\!-\!\tX_s^*|^2\big)\Big]
\leq C_T\E\Big[\sup_{s\in[0,t]}|X_s\!-\!\tX_s|^2\Big],
\end{align*}
as desired, because $(X^*,\tX^*)$ has the same law as $(X,\tX)$.
\vip
{\it Case 2: \eqref{m2c} holds.} Since
$\phi$ has linear growth, since $\phi,\gamma$ are Lipschitz continuous and since $\gamma$ is bounded, \eqref{lpd} implies that
\begin{align}\label{jab1}
\E\Big[|\Lambda_{\vmark,X}(t,\vmark,X)-\Lambda_{\vmark,\tX}(t,\vmark,\tX)|^2\Big]
\leq& C\bE[|U_t-\tU_t|^2]\notag\\
&+C\bE\Big[(1+|U_t|^2)
\big(|X_t-\tX_t|^2+|X_t^*-\tX_t^*|^2\big)\Big].
\end{align}
Moreover, \eqref{1jab2} still holds, while Proposition~\ref{estgamma}-(i) and~\eqref{m2c}
give, uniformly in $(\omega,\omega^*)\in \Omega\times\Omega^*$,
$$
\E^\dagger\big[\E^\#[1+|U_t|^2]\big]\leq C_T.
$$
Consequently,
\begin{align}
&\bE\Big[(1+|U_t|^2)
\big(|X_t-\tX_t|^2+|X_t^*-\tX_t^*|^2\big)\Big]\notag\\
&=\E\Big[\E^*\Big[\big(|X_t-\tX_t|^2+|X_t^*-\tX_t^*|^2\big)
\E^\dagger\big[\E^\#[1+|U_t|^2]\big]\Big]\Big]\notag\\
&\leq C_T\E\Big[\E^*\big[|X_t-\tX_t|^2+|X_t^*-\tX_t^*|^2\big]\Big]\notag\\
&\leq C_T\E\Big[\sup_{s\in[0,t]}|X_s-\tX_s|^2\Big]. \label{jab2}
\end{align}
The conclusion follows by inserting~\eqref{1jab2} and~\eqref{jab2}
into~\eqref{jab1}.
\end{proof}

\subsection{Proof of Theorem~\ref{mr}}

\begin{proof}[Proof of Theorem~\ref{mr}]
We first prove (i). We start with pathwise uniqueness and consider two solutions
$(\vmark,X)$ and $(\vmark,\tX)$ to the nonlinear SDE,
defined on the same probability space and driven by the same Brownian motion.
Since $b$ and $\sigma$ are bounded and Lipschitz continuous, we obtain, for all $t \in [0,T]$,
\begin{align*}
\E\Big[\sup_{s\in [0,t]}|X_s-\tX_s|^2\Big]\leq& C_T \int_0^t \E[|X_s-\tX_s|^2] \dd s
+ C_T \int_0^t \E[|\Lambda_{\vmark,X}(s,\vmark,X)-\Lambda_{\vmark,\tX}(s,\vmark,\tX)|^2] \dd s\\
\leq& C_T \int_0^t \E\Big[\sup_{u \in [0,s]} |X_u-\tX_u|^2\Big]\dd u
\end{align*}
by Proposition~\ref{estlambda}-(ii). One completes the uniqueness proof using the Gronwall lemma.
\vip

We next check existence by a Picard iteration. We fix $(\Omega,\cF,(\cF_t)_{t\geq 0},\PP)$
on which are defined $\vmark$ and $B$. We define $X^0=(X^0_t)_{t\geq 0}$ by $X^0_t=F(\vmark)$.
Then $X^0$ satisfies~\eqref{condA} with $A_T=0$.
Once $X^n=(X^n_t)_{t\geq 0}$ (continuous, adapted and satisfying~\eqref{condA}) is built, we set
\begin{equation}\label{pic}
X^{n+1}_t=F(\vmark)+\int_0^t b(X^{n}_s) \dd s +\int_0^t \sigma(X^{n}_s) \dd B_s
+ \int_0^t \Lambda_{\vmark,X^n}(s,\vmark,X^{n})\dd s.
\end{equation}
First, $X^{n+1}=(X^{n+1}_t)_{t\geq 0}$ is continuous and satisfies~\eqref{condA}
for some constant $A_T$ which does not depend on $n$, because
$b,\sigma$ and $\Lambda_{\vmark,X^n}$ are bounded, see Proposition~\ref{estlambda}-(i).
Indeed, when computing, for $0\leq s \leq t \leq s+1 \leq T+1$, the quantity $\E[|X^{n+1}_t-X^{n+1}_s|^2]$, the two drift terms are bounded in $L^2$ by $C_T(t-s)^2\leq C_T(t-s)$, while Itô's
isometry gives
$$
\E\Big[\Big|\int_s^t \sigma(X^n_r)\dd B_r\Big|^2\Big]\leq C_T(t-s).
$$
The same computation as in the uniqueness proof shows that for all $n\geq 1$, all $t\in [0,T]$,
\begin{align*}
\E\Big[\sup_{s\in [0,t]}|X^{n+1}_s-X^n_s|^2\Big]
\leq& C_T \int_0^t \E\Big[\sup_{u \in [0,s]} |X^n_u-X^{n-1}_u|^2\Big]\dd u.
\end{align*}
The usual Picard argument implies that there is a continuous adapted process $X=(X_t)_{t\geq 0}$ such that
\begin{equation}\label{conv}
\lim_n\E\Big[\sup_{t\in [0,T]}|X^n_t-X_t|^2\Big]=0.
\end{equation}
Since $X^n$ satisfies~\eqref{condA}
for some constant $A_T$ which does not depend on $n$, the limit $X$ satisfies~\eqref{condA}.
Taking the limit $n\to \infty$ in~\eqref{pic}, we find that
$(X_t)_{t\geq 0}$ solves the nonlinear SDE, because
\begin{gather*}
X^{n+1}_t\to X_t, \quad \int_0^t b(X^{n}_s)\dd s \to \int_0^t b(X_s)\dd s, \quad
\int_0^t \sigma(X^{n}_s)\dd B_s \to \int_0^t \sigma(X_s)\dd B_s,\\
\text{and} \quad \int_0^t \Lambda_{\vmark,X^n}(s,\vmark,X^{n})\dd s\to  \int_0^t \Lambda_{\vmark,X}(s,\vmark,X)\dd s
\end{gather*}
in $L^2(\Omega)$. The first three convergences follow from~\eqref{conv} and the
Lipschitz continuity of $b$ and $\sigma$, while the last one follows from
Proposition~\ref{estlambda}-(ii) and~\eqref{conv}.
\vip

We now prove (ii). The law of the solution $(\vmark,X)$ to the nonlinear SDE built
above by Picard iteration does not depend on the choice of the probability
space, nor on the choices of $\vmark$ and $B$. Indeed, one can check by induction
that the law of $(\vmark,X^n,B)$ does not depend on these choices for all $n\geq 1$,
because $\Lambda_{\vmark,X^n}$ is shorthand for
$\Lambda_{\mathrm{Law}(\vmark,X^n)}$, which depends only on the law of
$(\vmark,X^n)$; see Lemma~\ref{Lambda}. The limit $(\vmark,X,B)$
inherits this property. In particular, the law of $(\vmark,X)$
does not depend on the probability space nor on the choices of $\vmark$ and $B$.
\vip
Consider now, on some space $(\Omega,\cF,(\cF_t)_{t\ge 0},\PP)$, a solution
$X=(X_t)_{t\geq 0}$ to the nonlinear SDE corresponding to some $\cF_0$-measurable $\vmark\sim \vmarklaw$
and to some $m$-dimensional $(\cF_t)_{t\geq 0}$-Brownian motion
$B=(B_t)_{t\geq 0}$. On this space, consider the solution $(\vmark,\hat X)$ built by Picard iteration
with $\vmark$ and $B$. By the pathwise uniqueness proved in (i), $(\vmark,X)=(\vmark,\hat X)$ a.s.
The conclusion follows.
\end{proof}

\subsection{Proof of Theorem~\ref{mr2}}

\begin{proof}[Proof of Theorem~\ref{mr2}]
We first prove (i).
Recall that $(X^{i,N}_t)_{t\geq 0,i \in I_N}$ and $(U^{ij,N}_t)_{t\geq 0, i,j \in I_N,i\neq j}$ were
defined in~\eqref{s1}-\eqref{s2}.
For all $i \geq 1$, we denote by $(\bX^i_t)_{t\geq 0}$ the solution to the nonlinear SDE
corresponding to $\vmark_i$ and $B^i$, i.e.
\begin{equation}\label{bX}
\bX^i_t=F(\vmark_i)+\int_0^t b(\bX^i_s)\dd s + \int_0^t \sigma(\bX^i_s)\dd B^i_s+
\int_0^t \Lambda_{\vmark_i,\bX^i}(s,\vmark_i,\bX^i) \dd s,
\end{equation}
and, for $i,j \geq 1$ with $i\neq j$, we consider $(\bU^{ij}_t)_{t\geq 0}$ solving~\eqref{dfg2}.
As a strong solution, $\bX^i$ is
$\sigma(\vmark_i,B^i)$-measurable. Using moreover Theorem~\ref{mr}-(ii) and
Assumption~\ref{asr}, we conclude that
\begin{equation}\label{ind}
\text{the family $((\vmark_i,\bX^i), i\geq 1)$ is i.i.d. and independent of $((\emark_{ij},W^{ij}), i,j\geq 1,  i\neq j)$.}
\end{equation}

Since $\alpha$ and $\beta$ are bounded and Lipschitz continuous, we have,
for all $t\in [0,T]$,
\begin{align*}
\E\Big[\sup_{s\in [0,t]}|U^{12,N}_s-\bU^{12}_s|^2\Big] \leq& C_T \int_0^t \E[|X^{1,N}_s-\bX^1_s|^2+|X^{2,N}_s-\bX^2_s|^2
+ |U^{12,N}_s-\bU^{12}_s|^2] \dd s\\
\leq&  C_T \int_0^t \E[|X^{1,N}_s-\bX^1_s|^2
+ |U^{12,N}_s-\bU^{12}_s|^2] \dd s
\end{align*}
by exchangeability. Thanks to the Gronwall lemma, for all $t \in [0,T]$,
\begin{align}
\E\Big[\sup_{s\in [0,t]}|U^{12,N}_s-\bU^{12}_s|^2\Big] \leq& C_T \int_0^t \E[|X^{1,N}_s-\bX^1_s|^2] \dd s
\leq C_T \E\Big[\sup_{s\in [0,t]}|X^{1,N}_s-\bX^1_s|^2\Big]. \label{ttt1}
\end{align}

Since $b$ and $\sigma$ are Lipschitz continuous,
for all $t\in [0,T]$,
\begin{equation}\label{encun}
\E\Big[\sup_{s\in [0,t]}|X^{1,N}_s-\bX^1_s|^2\Big]\leq C_T \int_0^t\E[|X^{1,N}_s-\bX^1_s|^2] \dd s
+ C_T\int_0^t I_s \dd s,
\end{equation}
where, setting  $\mu=\mathrm{Law}(\vmark_1,\bX^1)$,
$$
I_t= \E\Big[\Big|\frac1{N-1}\sum_{j=2}^N \phi(U^{1j,N}_t)\gamma(X^{1,N}_t,X^{j,N}_t)
- \Lambda_\mu(t,\vmark_1,\bX^1)\Big|^2\Big].
$$
We write $I_t\leq 2J_t+2K_t$, where $J_t$ is the stability error and $K_t$ is
the empirical fluctuation error:
\begin{align*}
J_t=& \E\Big[\Big|\frac1{N-1}\sum_{j=2}^N \Big(\phi(U^{1j,N}_t)\gamma(X^{1,N}_t,X^{j,N}_t)
-\phi(\bU^{1j}_t)\gamma(\bX^1_t,\bX^j_t)\Big)\Big|^2\Big],\\
K_t=& \E\Big[\Big|\frac1{N-1}\sum_{j=2}^N\phi(\bU^{1j}_t)\gamma(\bX^1_t,\bX^j_t)
- \Lambda_\mu(t,\vmark_1,\bX^1)\Big|^2\Big].
\end{align*}
By the Cauchy--Schwarz inequality and since $\phi$ and $\gamma$ are bounded and Lipschitz continuous,
\begin{align*}
J_t\leq& \frac1{N-1}\sum_{j=2}^N\E\Big[\Big|\phi(U^{1j,N}_t)\gamma(X^{1,N}_t,X^{j,N}_t)
-\phi(\bU^{1j}_t)\gamma(\bX^1_t,\bX^j_t)\Big|^2\Big] \\
\leq&  \frac C{N-1}\sum_{j=2}^N\E\Big[ |U^{1j,N}_t-\bU^{1j}_t|^2+ |X^{1,N}_t-\bX^1_t|^2+|X^{j,N}_t-\bX^j_t|^2\Big]\\
\leq & \frac {C_T}{N-1}\sum_{j=2}^N\E\Big[\sup_{s\in [0,t]}(|X^{1,N}_s-\bX^{1}_s|^2+|X^{j,N}_s-{\bX^j_s}|^2)\Big]\\
\leq &C_T \E\Big[\sup_{s\in [0,t]} |X^{1,N}_s-\bX^{1}_s|^2\Big]
\end{align*}
by~\eqref{ttt1} and exchangeability.
Since $\bU^{1j}_t=\Gamma_t(\vmark_1,\bX^1,\vmark_j,\bX^j,\emark_{1j},W^{1j})$ by Lemma~\ref{mes},
$$
K_t=\E\Big[\Big|\frac1{N-1}\sum_{j=2}^N \phi(\Gamma_t(\vmark_1,\bX^1,\vmark_j,\bX^j,\emark_{1j},W^{1j}))\gamma(\bX^1_t,\bX^j_t)
- \Lambda_\mu(t,\vmark_1,\bX^1)\Big|^2\Big].
$$
By independence of $(\vmark_1,\bX^1)$ and
$(\vmark_j,\bX^j,\emark_{1j},W^{1j})_{j\geq 2}$, see~\eqref{ind}, 
The family 
$$
\big(\phi(\Gamma_t(\vmark_1,\bX^1,\vmark_j,\bX^j,\emark_{1j},W^{1j}))
\gamma(\bX^1_t,\bX^j_t)\big)_{j\geq2}
$$ 
is i.i.d. conditionally on $(\vmark_1,\bX^1)$. Moreover, by definition of $\Lambda_\mu$,
$$
\E_1\Big[
\phi(\Gamma_t(\vmark_1,\bX^1,\vmark_2,\bX^2,\emark_{12},W^{12}))
\gamma(\bX^1_t,\bX^2_t)\Big]=\Lambda_\mu(t,\vmark_1,\bX^1),
$$
where $\E_1= \E[\cdot | (\vmark_1,\bX^1)]$. Hence, abusively writing $\mathrm{Var}_{1} Z=
\E_{1}[|Z-\E_{1}[Z]|^2]$ for $Z$ valued in $\rr^d$,
$$
K_t= \E\Big[\frac 1{N-1} \mathrm{Var}_{1} \Big(
\phi(\Gamma_t(\vmark_1,\bX^1,\vmark_2,\bX^2,\emark_{12},W^{12}))
\gamma(\bX^1_t,\bX^2_t)\Big)\Big]
\leq \frac{||\phi||_\infty^2||\gamma||_\infty^2}{N-1}\leq \frac C N.
$$
We used that $N\geq 2$ in the last inequality.

\vip
All in all, we have proved that for all $t \in [0,T]$,
$$
I_t\leq 2J_t+2K_t\leq C_T \E\Big[\sup_{s\in [0,t]}|X^{1,N}_s-\bX^1_s|^2\Big] + \frac C N .
$$
Recalling~\eqref{encun}, we get, for all $t\in [0,T]$,
$$
\E\Big[\sup_{s\in [0,t]}|X^{1,N}_s-\bX^1_s|^2\Big]\leq C_T \int_0^t\E\Big[\sup_{u \in [0,s]}|X^{1,N}_u-\bX^1_u|^2\Big] \dd s
+ \frac {C_T} N.
$$
Using finally the Gronwall lemma, we find as desired that
$$
\E\Big[\sup_{t\in[0,T]}|X^{1,N}_t-\bX^1_t|^2\Big] \leq \frac{C_T}N, \quad \text{whence also}\quad
\E\Big[\sup_{t\in[0,T]}|U^{12,N}_t-\bU^{12}_t|^2\Big] \leq \frac{C_T}N
$$
by~\eqref{ttt1}.

\vip

We now prove (ii). We adopt the same notation as above and observe that we still have~\eqref{ind}.
We fix $T>0$ and divide the proof into several steps.
\vip
{\it Step 1 : Gaussian moment for the edge.} 
There are $\theta_T>0$ and $C_T>0$ such that
$$
\E\Big[\exp\Big(2 \theta_T \sup_{t\in [0,T]} (\bU^{12}_t)^2\Big)\Big] \leq C_T.
$$
Recall~\eqref{dfg2}: setting $G_{12}=G(\vmark_1,\vmark_2,\emark_{12})$ and
$R_t=\int_0^t \beta(\bX^1_s,\bX^2_s,\bU^{12}_s)\dd W_s^{12}$, 
$$
\sup_{t\in[0,T]}|\bU^{12}_t|\leq |G_{12}|+T||\alpha||_\infty+R_T^*,
\quad \hbox{where}\quad R_T^*=\sup_{t\in[0,T]}|R_t|.
$$
By~\eqref{meu}, $G_{12}$ has a Gaussian moment.
By the Dubins-Schwarz theorem, see e.g. Revuz-Yor~\cite[Theorem~1.6 p~181]{ry},
we can find a Brownian motion $M$ such that a.s., for all $t\geq 0$, $R_t=M_{A_t}$,
with $A_t=\int_0^t \beta^2(\bX^1_s,\bX^2_s,\bU^{12}_s)\dd s$. Thus $R_T^*\leq \sup_{t\in[0,||\beta||_\infty^2 T]} |M_t|$.
Hence $R_T^*$ has a Gaussian moment. Choosing $\theta_T>0$ small
enough and using $(x+y+z)^2\leq 3(x^2+y^2+z^2)$ together with the Cauchy--Schwarz inequality, we get the conclusion.

\vip

{\it Step 2: Empirical fluctuation.} Set
$\mu=\mathrm{Law}(\vmark_1,\bX^1)$. There is $C_T>0$ such that for all
$t\in [0,T]$,
$$
\Delta_t:=\E\Big[\Big|\frac1{N-1}\sum_{j=2}^N \phi(\bU^{1j}_t)\gamma(\bX^1_t,\bX^j_t)
-\Lambda_\mu(t,\vmark_1,\bX^1)\Big|\Big] \leq \frac{C_T}{N^{\frac 1 2}}.
$$
As in the proof of (i) (study of $K_t$)
we have $\Delta_t=\E[E_t]$, where, setting $\E_1=\E[\cdot|(\xi_1,\bX^1)]$,
$$
E_t=\E_{1}\Big[\Big|\frac1{N-1}\sum_{j=2}^N
\phi(\Gamma_t(\vmark_1,\bX^1,\vmark_j,\bX^j,\emark_{1j},W^{1j}))
\gamma(\bX^1_t,\bX^j_t)-\Lambda_\mu(t,\vmark_1,\bX^1)\Big|\Big].
$$
We have
\begin{align*}
E_t\leq &\E_{1}\Big[\Big|\frac1{N-1}\sum_{j=2}^N
\phi(\Gamma_t(\vmark_1,\bX^1,\vmark_j,\bX^j,\emark_{1j},W^{1j}))
\gamma(\bX^1_t,\bX^j_t)-\Lambda_\mu(t,\vmark_1,\bX^1)\Big|^2\Big]^{\frac12}\\
\leq & \Big(\frac 1{N-1} \mathrm{Var}_{1} \;\Big(
\phi(\Gamma_t(\vmark_1,\bX^1,\vmark_2,\bX^2,\emark_{12},W^{12}))
\gamma(\bX^1_t,\bX^2_t)\Big)
\Big)^{\frac12}
\end{align*}
as in the proof of (i). Thus, recalling that $\Gamma_t(\vmark_1,\bX^1,\vmark_2,\bX^2,\emark_{12},W^{12})=\bU^{12}_t$, for all $N\geq 2$,
$$
\Delta_t\leq \sqrt{\frac2{N}} \E[(\phi(\bU^{12}_t)\gamma(\bX^1_t,\bX^2_t))^2 ]
\leq \frac C{N^{\frac12}} \E[1+(\bU^{12}_t)^2] \leq \frac {C_T}{N^{\frac12}}.
$$
We used that $\phi$ has at most linear growth, that $\gamma$ is bounded, and
that $\sup_{[0,T]}\E[(\bU^{12}_t)^2]<\infty$.

\vip
{\it Step 3 : Notation.}
 For $i,j\in I_N$ with $i\neq j$ and $0\leq s\leq t$, set
$$
Y^{i,N}_{s,t}=\sup_{u\in [s,t]} |X^{i,N}_u-\bX^i_u|
\quad \text{and} \quad
Z^{ij,N}_{s,t}=\sup_{u\in [s,t]} |U^{ij,N}_u-\bU^{ij}_u|.
$$

{\it Step 4. Edge stability.}
In this step we establish a short-time estimate for the edge error $Z^{12,N}$.
Fix $0\leq S\leq S'\leq T$ and $t\in[S,S']$. Recalling~\eqref{s2} and~\eqref{dfg2}, we see that
$Z^{12,N}_{S,t} \leq Z^{12,N}_{S,S}+ I^1_{S,t}+I^2_{S,t}$, where
\begin{align*}
I^1_{S,t}=&\int_S^t |\alpha(X^{1,N}_s,X^{2,N}_s,U^{12,N}_s)-
\alpha(\bX^{1}_s,\bX^{2}_s,\bU^{12}_s)| \dd s,\\
I^2_{S,t}=&\sup_{r\in [S,t]} \Big|\int_S^r (\beta(X^{1,N}_s,X^{2,N}_s,U^{12,N}_s)-
\beta(\bX^{1}_s,\bX^{2}_s,\bU^{12}_s))\dd W^{12}_s \Big|.
\end{align*}
Using that $\alpha$ is Lipschitz continuous, we find
$$
\E[I^1_{S,t}]\leq C(t-S) \E[Y^{1,N}_{S,t}+ Y^{2,N}_{S,t}+Z^{12,N}_{S,t}] \leq C(S'-S)\E[Y^{1,N}_{S,t}+Z^{12,N}_{S,t}]
$$
by exchangeability. Using the Burkholder--Davis--Gundy inequality and that $\beta$ is Lipschitz continuous,
\begin{align*}
\E[I^2_{S,t}]\leq C \E\Big[\Big(\int_S^t (\beta(X^{1,N}_s,X^{2,N}_s,U^{12,N}_s)-
\beta(\bX^{1}_s,\bX^{2}_s,\bU^{12}_s))^2 \dd s \Big)^{\frac 12}\Big]
\leq  C (S'-S)^{\frac12}\E[Y^{1,N}_{S,t}+Z^{12,N}_{S,t}]
\end{align*}
by exchangeability again. Combining the bounds on $I^1_{S,t}$ and $I^2_{S,t}$, we get,
for all $0\leq S \leq t \leq S' \leq T$,
\begin{align*}
\E[Z^{12,N}_{S,t}] \leq & \E[Z^{12,N}_{S,S}]
+ \kappa_{1} [S'-S+(S'-S)^{\frac12}]\E[Y^{1,N}_{S,t}+Z^{12,N}_{S,t}]
\end{align*}
for some $\kappa_{1}>0$ depending only on $\alpha$ and $\beta$.

\vip
{\it Step 5: Particle stability.} Here we establish a short-time estimate
for $Y^{1,N}$. Fix $0\leq S\leq S'\leq T$ and $t\in[S,S']$.
Recalling~\eqref{s1} and~\eqref{bX}, we have
$Y^{1,N}_{S,t}\leq Y^{1,N}_{S,S}+J^1_{S,t}+J^2_{S,t}+J^3_{S,t}$, where
\begin{gather*}
J^1_{S,t}=\int_S^t |b(X^{1,N}_s)-b(\bX^1_s)|\dd s, \quad
J^2_{S,t}=\sup_{r\in [S,t]} \Big|\int_S^r (\sigma(X^{1,N}_s)-\sigma(\bX^{1}_s))\dd B^1_s\Big|,\\
J^{3}_{S,t}= \int_S^t \Big|\frac1{N-1}\sum_{j=2}^N \phi(U^{1j,N}_s)\gamma(X^{1,N}_s,X^{j,N}_s)
- \Lambda_\mu(s,\vmark_1,\bX^1)\Big| \dd s.
\end{gather*}
The terms $J^1_{S,t}$ and $J^2_{S,t}$ are the drift and martingale contributions
coming from the single-particle coefficients $b$ and $\sigma$, while $J^3_{S,t}$
is the interaction contribution. Exactly as in Step~4, using the Lipschitz
continuity of $b$ and $\sigma$ and the Burkholder--Davis--Gundy inequality, we find that
$$
\E[J^1_{S,t}+J^2_{S,t}] \leq C [S'-S+(S'-S)^{\frac12}] \E[Y^{1,N}_{S,t}+Z^{12,N}_{S,t}].
$$
Recalling Step~2, we write
\begin{align*}
\E[J^3_{S,t}] \leq& \int_S^t \E\Big[\Big|\frac1{N-1}\sum_{j=2}^N \Big(\phi(U^{1j,N}_s)\gamma(X^{1,N}_s,X^{j,N}_s)
- \phi(\bU^{1j}_s)\gamma(\bX^{1}_s,\bX^{j}_s)  \Big)\Big|  \Big]\dd s + \int_S^t \Delta_s \dd s \\
\leq & \int_S^t \E[|\phi(U^{12,N}_s)\gamma(X^{1,N}_s,X^{2,N}_s)
- \phi(\bU^{12}_s)\gamma(\bX^{1}_s,\bX^{2}_s)|] \dd s + C_T N^{-\frac 12}
\end{align*}
by exchangeability and Step~2. We now set 
$H_s^N=(|X^{1,N}_s-\bX^1_s|\wedge 1)+(|X^{2,N}_s-\bX^2_s|\wedge 1)$.
Since $\gamma$ and $\phi$ are Lipschitz continuous, with $\gamma$
bounded, for $t\in [S,S']$,
\begin{align*}
\E[J^3_{S,t}] \leq& C\int_S^t \E[(1+|\bU^{12}_s|)H_s^N
+|U^{12,N}_s-\bU^{12}_s|] \dd s + C_T N^{-\frac 12} \\
\leq & C \int_S^t \E[|\bU^{12}_s|H_s^N] \dd s
+ C(S'\!-\!S)\E[Y^{1,N}_{S,t}+Z^{12,N}_{S,t}]+C_T N^{-\frac 12}.
\end{align*}
Combining these estimates (and using exchangeability),
we get, for all $0\leq S \leq t \leq S' \leq T$,
\begin{align*}
\E[Y^{1,N}_{S,t}] \leq & \E[Y^{1,N}_{S,S}]
+ C_T N^{-\frac12}+ \kappa_{2} [S'-S+(S'-S)^{\frac12}] \E[Y^{1,N}_{S,t}+Z^{12,N}_{S,t}]
+ \kappa_{2} \int_S^t\E[|\bU^{12}_s|H_s^N]\dd s
\end{align*}
for some $\kappa_{2}>0$ depending only on $b$, $\sigma$, $\gamma$ and $\phi$.

\vip
{\it Step 6: Truncation and recursion.} Gathering Steps~4 and~5, we find $\kappa_{3}>0$ such that for all
$0\leq S\leq S'\leq T$, all $t\in [S,S']$,
\begin{align*}
\E[ Y^{1,N}_{S,t}+Z^{12,N}_{S,t}]\leq&  \E[Y^{1,N}_{S,S}+Z^{12,N}_{S,S}]
+ C_T N^{-\frac12}+ \kappa_{3} [S'-S+(S'-S)^{\frac12}] \E[Y^{1,N}_{S,t}+Z^{12,N}_{S,t}]\\
&+ \kappa_{3} \int_S^t\E[|\bU^{12}_s|H_s^N]\dd s.
\end{align*}
Let $\delta>0$ such that $\kappa_3(\delta+\delta^{\frac12})=\frac12$ and
$\kappa_{4}=2\kappa_{3}$. Then for all
$0\leq S \leq S' \leq T \land(S+\delta)$, all $t \in [S,S']$,
\begin{align*}
\E[ Y^{1,N}_{S,t}+Z^{12,N}_{S,t}]\leq&  2\E[Y^{1,N}_{S,S}+Z^{12,N}_{S,S}]
+ C_T N^{-\frac12}+\kappa_4 \int_S^t\E[|\bU^{12}_s|H_s^N]\dd s.
\end{align*}
The only remaining difficulty is the unbounded factor $\bU^{12}$ in
the last integral. We handle it by truncating at a suitable level $A$.
For any $A>0$ and $s\geq S$, using exchangeability and Step~1,
\begin{align*}
\E[|\bU^{12}_s|H_s^N]
&\leq A\E[|X^{1,N}_s-\bX^1_s|+|X^{2,N}_s-\bX^2_s|]
+2\E[|\bU^{12}_s|\indiq_{\{|\bU^{12}_s|>A\}}]\\
&\quad\leq 2A\E[Y^{1,N}_{S,s}] + C_T e^{-\theta_T A^2}.
\end{align*}
In the last line, we used that
$u\indiq_{\{u>A\}}\leq C_T e^{-\theta_T A^2}e^{2\theta_T u^2}$ and the Gaussian moment
from Step~1. Plugging this into the preceding inequality, we get that
for all $0\leq S \leq S' \leq T \land(S+\delta)$, all $t \in [S,S']$ and all $A>0$,
\begin{align*}
\E[Y^{1,N}_{S,t}+Z^{12,N}_{S,t}] \leq& 2\E[Y^{1,N}_{S,S}+Z^{12,N}_{S,S}]+ C_T (N^{-\frac12} +e^{-\theta_T A^2})
+ 2A \kappa_4 \int_S^t\E[Y^{1,N}_{S,s}+Z^{12,N}_{S,s}]\dd s.
\end{align*}
Using the Gronwall lemma, we find that for all $0\leq S \leq S' \leq T \land(S+\delta)$, for all $A>0$,
\begin{align*}
\E[Y^{1,N}_{S,S'}+Z^{12,N}_{S,S'}] \leq& \Big(2\E[Y^{1,N}_{S,S}+Z^{12,N}_{S,S}]+ C_T (N^{-\frac12} +e^{-\theta_T A^2})
\Big)e^{2\kappa_4 A \delta}.
\end{align*}
Choosing $A=(\frac{\log N}{2\theta_T})^{\frac12}$,
this gives, for all $0\leq S \leq S' \leq T \land(S+\delta)$,
\begin{equation}\label{thet}
\begin{aligned}
\E[Y^{1,N}_{S,S'}+Z^{12,N}_{S,S'}]
&\leq \Big(2\E[Y^{1,N}_{S,S}+Z^{12,N}_{S,S}]+ C_TN^{-\frac12}\Big) K_{N,T},\\
\text{where} \quad K_{N,T}
&=\exp\Big(2\kappa_4 \delta \Big(\frac{\log N}{2\theta_T}\Big)^{\frac12}\Big).
\end{aligned}
\end{equation}

{\it Step 7: Conclusion.} Set
$k=\lceil T/\delta\rceil$ and $t_\ell=\ell\delta\land T$, for $\ell=0,\dots,k$.
Set $u_{N,0}=0$ and, for $\ell=0,\dots,k-1$,
$$
u_{N,\ell+1}=\E[Y^{1,N}_{t_\ell,t_{\ell+1}}+Z^{12,N}_{t_\ell,t_{\ell+1}}].
$$
By~\eqref{thet}, for all $\ell=0,\dots,k-1$,
$$
u_{N,\ell+1} \leq (2u_{N,\ell}+C_T N^{-\frac12})K_{N,T}.
$$
Indeed, the error at the left endpoint of the $\ell$-th block is bounded by
$u_{N,\ell}$. We classically conclude that, modifying the value of $C_T$,
$$
u_{N,\ell} \leq C_T K_{N,T}^{k}N^{-\frac12}
$$
for all $\ell=1,\dots,k$.
All in all, modifying again the value of $C_T$,
$$
\E\Big[\sup_{s\in [0,T]} (|X^{1,N}_s\!-\!\bX^1_s|\!+\!|U^{12,N}_s\!-\!\bU^{12}_s|)\Big]\leq\sum_{\ell=1}^{k}u_{N,\ell}
\leq C_T K_{N,T}^{k} N^{-\frac12}.
$$
Since
$$
K_{N,T}^{k}=\exp\Big(2\kappa_4 k \delta \Big(\frac{\log N}{2\theta_T}\Big)^{\frac12}\Big),
$$
the conclusion of (ii) follows.

\vip

We finally prove (iii). Under the standing alternative in the statement,
either the assumption of Theorem~\ref{mr2}-(i) or that of Theorem~\ref{mr2}-(ii)
holds.
We let $P_N:=\{(i,j) \in I_N,i\neq j\}$. As already seen, $\bX^i$ is $\sigma(\vmark_i,B^i)$-measurable and $\bU^{ij}$ is $\sigma(\vmark_i,B^i,\vmark_j,B^j,\emark_{ij},W^{ij})$-measurable. Using Theorem~\ref{mr} and Assumption~\ref{asr},
we see that
all the random variables $\mathcal Z^{ij}_t=(\bX^{i}_t,\bX^{j}_t,\bU^{ij}_t)$ are $f_t$-distributed
and that for $C\subset P_N$, the family $(\mathcal Z^{ij}_t, (i,j)\in C)$ is i.i.d. as soon as
for all $(i,j),(i',j') \in C$ distinct, the four numbers
$i,j,i',j'$ are distinct. We set
$$
\mu^N_t=\frac1{N(N-1)}\sum_{(i,j) \in P_N}\delta_{(X^{i,N}_t,X^{j,N}_t,U^{ij,N}_t)}\quad \hbox{and}\quad
\bar \mu^N_t=\frac1{N(N-1)}\sum_{(i,j) \in P_N}\delta_{(\bX^{i}_t,\bX^{j}_t,\bU^{ij}_t)}.
$$
We write $\cW_1(\mu^N_t,f_t) \leq \cW_1(\mu^N_t,\bar \mu^N_t)+\cW_1(\bar \mu^N_t,f_t)$.
We classically have
$$
\E[ \cW_1(\mu^N_t,\bar \mu^N_t)]\leq \frac1{N(N-1)}\sum_{(i,j) \in P_N}\E[|X^{i,N}_t-\bX^i_t|
+|X^{j,N}_t-\bX^j_t|+|U^{ij,N}_t-\bU^{ij}_t|] \leq \frac{C_T}{N^{\frac 1{2d+1}}}
$$
by Theorem~\ref{mr2}-(i) (since $N^{-\frac12}\leq N^{-\frac 1{2d+1}}$) or Theorem~\ref{mr2}-(ii)
(since $N^{-\frac12}\exp(C_T\sqrt{\log N})\leq C_T N^{-\frac 1{2d+1}}$), using
Cauchy--Schwarz in the first case to pass from the $L^2$ estimate to $L^1$.
\vip

We next write $P_N= \bigsqcup_{k=0}^{2N-1} P_N^k$,
where, for $k=0,\dots, N-1$,
$$
P_N^k=\{(i,j) \in I_N : i+j=k \!\!\mod N, i<j\}\;\; \text{and} \;\;
P_N^{N+k}=\{(i,j) \in I_N : i+j=k \!\!\mod N, i>j\}.
$$
Easy considerations show that for all $k=0,\dots,2N-1$, all $(i,j),(i',j') \in P_N^k$ with $(i,j)\neq(i',j')$,
the four numbers $i,i',j,j'$ are distinct, so that the family $(\mathcal Z^{ij}_t, (i,j)\in P_N^k)$ is i.i.d.
Moreover,
$$
\#(P_N^k)=\begin{cases}
\frac{N-1}2 & \text{if $N$ is odd},\\
\frac N2& \text{if $N$ is even and $k$ is odd},\\
\frac N2-1& \text{if $N$ is even and $k$ is even}.
\end{cases}
$$
Assuming now that $N\geq 3$ is odd, we have
$$
\bar \mu^N_t=\frac1{2N} \sum_{k=0}^{2N-1} \bar \mu^{N,k}_t,\quad \text{where}\quad
\bar \mu^{N,k}_t=\frac 2{N-1}\sum_{(i,j) \in P_N^k}\delta_{(\bX^{i}_t,\bX^{j}_t,\bU^{ij}_t)}.
$$
Using the joint convexity of $\cW_1$ in its measure arguments, we have
$$
\cW_1(\bar \mu^N_t,f_t)\leq \frac1{2N} \sum_{k=0}^{2N-1}\cW_1(\bar \mu^{N,k}_t,f_t).
$$
For each $k=0,\dots,2N-1$, $\bar \mu^{N,k}_t$ is the empirical measure of $\frac{N-1}2$ i.i.d. random variables
with law $f_t\in \cP(\rr^d\times\rr^d\times\rr)$. Using the defining equations,
Assumption~\ref{asm}, the standing alternative in the theorem, and the boundedness
of the coefficients, one easily checks that for all $T>0$,
$$
\sup_{t\in [0,T]}\E[|\bX^1_t|^2+|\bX^2_t|^2+|\bU^{12}_t|^2]<\infty.
$$
By~\cite[Theorem 1]{fg} with $(p,d,q)$ replaced by $(1,2d+1,2)$, for all $t\in [0,T]$, all $k=0,\dots,2N-1$,
$$
\E[\cW_1(\bar \mu^{N,k}_t,f_t)] \leq C_T N^{-\frac 1{2d+1}}.
$$
All this shows that if $N\geq 3$ is odd,
$$
\E[\cW_1(\bar \mu^N_t,f_t)]\leq C_T N^{-\frac 1{2d+1}},\quad \text{whence}
\quad \E[\cW_1(\mu^N_t,f_t)]\leq C_T N^{-\frac 1{2d+1}}.
$$
The case where $N\geq 4$ is even is treated similarly, with light complications.
Taking the supremum over $t\in[0,T]$ gives the claim.
\end{proof}

\section{Measurability}\label{smes}

We first prove Lemma~\ref{mes}, using arguments found in Karandikar~\cite{k}.

\begin{proof}[Proof of Lemma~\ref{mes}]
We set $\cX=\vmarkspace\times \cC_d\times \vmarkspace\times \cC_d\times\emarkspace
\times \cC_1$.
\vip
{\it Step 1: Euler scheme.}
Fix $n\geq 1$.
For $(\pmark,x,\qmark,y,\rmark,w)\in\cX$, we denote by
$(\Gamma^n_t(\pmark,x,\qmark,y,\rmark,w))_{t\geq 0}$ the Euler scheme with step $2^{-n}$ for the
equation $u_t=G(\pmark,\qmark,\rmark)+\int_0^t \alpha(x_s,y_s,u_s) \dd s + \int_0^t \beta(x_s,y_s,u_s) \dd w_s$:
we set $q_0=G(\pmark,\qmark,\rmark)$ and, for $k\geq 0$,
\begin{align*}
q_{(k+1)2^{-n}}=q_{k2^{-n}}+ 2^{-n}\alpha(x_{k2^{-n}},y_{k2^{-n}},q_{k2^{-n}})+
(w_{(k+1)2^{-n}}-w_{k2^{-n}})\beta(x_{k2^{-n}},y_{k2^{-n}},q_{k2^{-n}}).
\end{align*}
Finally, for all $t\geq 0$, choose $k=\lfloor 2^nt \rfloor$ and set
$$
\Gamma^n_t(\pmark,x,\qmark,y,\rmark,w)= q_{k2^{-n}} + \Big(t-k2^{-n}\Big)\alpha(x_{k2^{-n}},y_{k2^{-n}},q_{k2^{-n}})+
(w_{t}-w_{k2^{-n}})\beta(x_{k2^{-n}},y_{k2^{-n}},q_{k2^{-n}}).
$$
The map $\Gamma^n : \cX\to \cC_1$ is measurable, because it is a continuous function of $(G(\pmark,\qmark,\rmark),x,y,w)$.
By construction, it holds that for all $(\pmark,x,\qmark,y,\rmark,w)\in \cX$, all $t\geq 0$,
\begin{equation}\label{prog2}
\Gamma^n_t(\pmark,x,\qmark,y,\rmark,w)=\Gamma^n_t(\pmark,(x_{s\land t})_{s\geq 0},\qmark,y,\rmark,w).
\end{equation}

{\it Step 2: Convergence to the edge SDE.}
Here we prove that for any space $(\Omega,\cF,(\cF_t)_{t\ge 0},\PP)$, any independent 
$\cF_0$-measurable triple $(\vmark,\nmark,\emark)$ with $\vmark,\nmark\sim \vmarklaw$ and $\emark\sim \emarklaw$,
any pair of continuous $(\cF_t)_{t\geq 0}$-adapted processes $X=(X_t)_{t\geq 0}$ and $Y=(Y_t)_{t\geq 0}$ valued in $\rr^d$
both satisfying~\eqref{condA},
any $1$-dimensional $(\cF_t)_{t\geq 0}$-Brownian motion
$W=(W_t)_{t\geq 0}$, the pathwise unique solution $U=(U_t)_{t\geq 0}$ to~\eqref{dfg}
satisfies
\begin{equation}\label{cps}
\text{for all $T>0$,}\quad \lim_n \sup_{t\in [0,T]} |U_t-\Gamma^n_t(\vmark,X,\nmark,Y,\emark,W)|^2=0 \quad \text{a.s.}
\end{equation}

Since $\alpha$ and $\beta$ are bounded, there is a constant $C>0$ such that
for all $0\leq s<t<s+1$,
\begin{equation}\label{condAU}
\E[|U_t-U_s|^2] \leq C (t-s).
\end{equation}
We next set $U^n_t=\Gamma^n_t(\vmark,X,\nmark,Y,\emark,W)$, 
which classically satisfies, introducing $\rho_n(t)=2^{-n}\lfloor 2^nt\rfloor$,
$$
U_t^n=G(\vmark,\nmark,\emark)+ \int_0^t \alpha(X_{\rho_n(s)},Y_{\rho_n(s)},U^n_{\rho_n(s)})\dd s
+ \int_0^t \beta(X_{\rho_n(s)},Y_{\rho_n(s)},U^n_{\rho_n(s)})\dd W_s.
$$
Since $\alpha$ and $\beta$ are bounded and Lipschitz continuous, we have, for all $t\in [0,T]$,
\begin{align*}
\E\Big[\sup_{s\in [0,t]}|U_s-U^n_s|^2\Big]\leq& C_T \int_0^t \E\Big[|X_s-X_{\rho_n(s)}|^2
+|Y_s-Y_{\rho_n(s)}|^2+ |U_s-U^n_{\rho_n(s)}|^2
\Big] \dd s \\
\leq & C_T \e_{T,n} + C_T \int_0^t \E\Big[\sup_{u\in [0,s]}|U_u-U^n_u|^2\Big] \dd s,
\end{align*}
where
$$
\e_{T,n}:=\int_0^T \E\Big[|X_s-X_{\rho_n(s)}|^2
+|Y_s-Y_{\rho_n(s)}|^2+ |U_s-U_{\rho_n(s)}|^2
\Big] \dd s \leq C_T 2^{-n}
$$
by~\eqref{condAU}, since $X$ and $Y$ satisfy~\eqref{condA} and since
$0\leq s-\rho_n(s)\leq 2^{-n}$ for all $s\geq 0$. By Gronwall's lemma,
$$
\E\Big[\sup_{s\in [0,T]}|U_s-U^n_s|^2\Big]\leq C_T 2^{-n}.
$$
Thus the series $\sum_{n\geq 1}\sup_{s\in [0,T]}|U_s-U^n_s|^2$ a.s. converges, whence $\lim_n \sup_{s\in [0,T]}|U_s-U^n_s|^2=0$ a.s.
Taking the intersection of the corresponding probability-one events over integer
$T\geq 1$, we also have $U^n\to U$ in $(\cC_1,\delta)$ a.s.
where $\delta(z,z')=\sum_{\ell\geq 1}2^{-\ell}[1\land \sup_{t\in[0,\ell]} |z(t)-z'(t)|]$ classically metrizes the
topology of uniform convergence on compact time intervals.

\vip

{\it Step 3: Measurability of the limit map.}
Since $(\cC_1,\delta)$ is complete, the set
$$
\cA=\{(\pmark,x,\qmark,y,\rmark,w)\in\cX :
\lim_n \Gamma^n(\pmark,x,\qmark,y,\rmark,w) \text{ exists in }(\cC_1,\delta)\}
$$
is measurable.
We then set
$$
\Gamma(\pmark,x,\qmark,y,\rmark,w)= \lim_n \indiq_{\{(\pmark,x,\qmark,y,\rmark,w)\in \cA\}}\Gamma^n(\pmark,x,\qmark,y,\rmark,w)
$$
for all $(\pmark,x,\qmark,y,\rmark,w)\in \cX$. This map $\Gamma : \cX\to \cC_1$ is measurable.

\vip
{\it Step 4: Identification with the SDE.}
Consider a space $(\Omega,\cF,(\cF_t)_{t\ge 0},\PP)$, a pair of independent
$\vmarklaw$-distributed $\cF_0$-measurable random variables $\vmark,\nmark$ valued in
$\vmarkspace$, an $\cF_0$-measurable $\emarkspace$-valued
random variable $\emark$ independent of $(\vmark,\nmark)$ and with law
$\emarklaw$, a pair of continuous $(\cF_t)_{t\geq 0}$-adapted processes
$X=(X_t)_{t\geq 0}$ and $Y=(Y_t)_{t\geq 0}$ valued in $\rr^d$ both satisfying~\eqref{condA},
a $1$-dimensional $(\cF_t)_{t\geq 0}$-Brownian motion
$W=(W_t)_{t\geq 0}$, and the pathwise unique solution $U=(U_t)_{t\geq 0}$ to~\eqref{dfg}.
Consider the event $A=\{(\vmark,X,\nmark,Y,\emark,W) \in \cA\}$, which has probability $1$ by Step~2.
On $A$, we have $U=\Gamma(\vmark,X,\nmark,Y,\emark,W)$ (in the sense that for all $t\geq 0$, $U_t=\Gamma_t(\vmark,X,\nmark,Y,\emark,W)$)  again by Step~2
and by definition of $\Gamma$.

\vip

{\it Step 5: Non-anticipativity.}
With the same notation as in Step~4, on $A$, we have that for all $t\geq 0$,
$$
\begin{aligned}
\Gamma_t(\vmark,X,\nmark,Y,\emark,W)
&=\lim_n\Gamma_t^n(\vmark,X,\nmark,Y,\emark,W)\\
&=\lim_n \Gamma_t^n(\vmark,(X_{s\land t})_{s\geq 0},\nmark,Y,\emark,W)\\
&=\Gamma_t(\vmark,(X_{s\land t})_{s\geq 0},\nmark,Y,\emark,W)
\end{aligned}
$$
by~\eqref{prog2}. The last equality is legitimate because 
$(X_{s\land t})_{s\geq 0}$ satisfies the conditions of Steps~2 and~3.
This implies~\eqref{prog}.
\end{proof}

We now turn to Lemma~\ref{Lambda}, using the non-anticipativity of $\Gamma$
from Lemma~\ref{mes} and the domination estimate from
Proposition~\ref{estgamma}.

\begin{proof}[Proof of Lemma~\ref{Lambda}.]
Fix an admissible probability measure $\mu$ and an admissible pair $(\vmark,X)$
defined on $(\Omega,\cF,(\cF_t)_{t\geq0},\PP)$.
Let $(\Omega^*,\cF^*,(\cF_t^*)_{t\geq0},\PP^*)$ carry an admissible pair
$(\nmark,Y)$ with law $\mu$, let
$(\Omega^\dagger,\cF^\dagger,\PP^\dagger)$ carry an
$\emarkspace$-valued random variable $\emark$ with
law $\emarklaw$, and let
$(\Omega^\#,\cF^\#,(\cF_t^\#)_{t\geq0},\PP^\#)$ carry a one-dimensional Brownian
motion $W$. These auxiliary spaces are taken independent of $(\Omega,\cF,(\cF_t)_{t\geq0},\PP)$ and of each other. As in the proof of Lemma~\ref{estlambda}, we have
$$
\Lambda_{\mu}(t,\xi,X)=\E^*\Big[\E^\dagger\Big[\E^\#\Big[
\phi(\Gamma_t(\vmark,X,\nmark,Y,\emark,W)) \gamma(X_t,Y_t)
\Big]\Big]\Big].
$$
For all $T>0$, there is a constant $C_T$ such that a.s.,
\begin{equation}\label{le}
\Psi_{\mu,T}(\vmark,X):=
\E^*\Big[\E^\dagger\Big[\E^\#\Big[
\sup_{t\in [0,T]} |\phi(\Gamma_t(\vmark,X,\nmark,Y,\emark,W))|
\Big]\Big]\Big] \leq C_T.
\end{equation}
If $\phi$ is bounded, this is obvious. Else, this follows from 
Proposition~\ref{estgamma}-(i), the fact that $\phi$ has at most linear growth and~\eqref{m2c}.
\vip
Point (i), i.e. continuity of
$(\Lambda_\mu(t,\vmark,X))_{t\geq 0}$, follows from dominated convergence by~\eqref{le},  the
continuity and boundedness of $\gamma$, the continuity of $\phi$ and
the facts that $\Gamma$ is valued in $\cC_1$ while $X$ and $Y$ are valued in $\cC_d$.
\vip
For (ii), i.e. $(\cF_t)_{t\geq 0}$-adaptedness of $(\Lambda_\mu(t,\vmark,X))_{t\geq 0}$, it suffices to use~\eqref{prog}: we have
$$
\Lambda_{\mu}(t,\xi,X)=\E^*\Big[\E^\dagger\Big[\E^\#\Big[
\phi(\Gamma_t(\vmark,(X_{s\land t})_{s\geq 0},\nmark,Y,\emark,W)) \gamma(X_t,Y_t)
\Big]\Big]\Big].
$$
which is of course $\cF_t$-measurable.
\end{proof}

\bibliographystyle{plain}
\bibliography{pcp}
\end{document}